\documentclass[12pt]{amsart}
\usepackage{euler, amsfonts, amssymb, latexsym, epsfig}
\usepackage[bookmarks, bookmarksdepth=2, colorlinks=true, linkcolor=blue, citecolor=blue, urlcolor=blue]{hyperref}
\usepackage{cancel,mathtools}
\usepackage[enableskew]{youngtab}
\usepackage{libertine}

\newcommand\junk[1]{}
\DeclareMathOperator{\codim}{codim}
\DeclareMathOperator\sign{sign}

\newcommand{\arxiv}[1]{\href{http://arxiv.org/abs/#1}{{\tt arXiv:#1}}}

\newcommand\lie[1]{{\mathfrak #1}}

\newcommand\iso{\mathrel{\cong}}

\newcommand\RR{\mathbb R}
\newcommand\tensor{{\otimes}}

\newcommand\calO{{\mathcal O}}

\newcommand\calD{{\mathcal D}}
\newcommand\calF{{\mathcal F}}
\newtheorem{Theorem}{Theorem}

\newtheorem{Lemma}{Lemma}

\newtheorem*{Corollary*}{Corollary}
\newtheorem*{Lemma*}{Lemma}

\newtheorem*{Theorem*}{Theorem}
\theoremstyle{remark}
\newtheorem{Example}{Example}

\newcommand\onto{\mathrel{\twoheadrightarrow}}

\newcommand\into{\mathrel{\hookrightarrow}}

\newcommand\Union{\bigcup}

\newcommand\PP{{\mathbb P}}
\newcommand\CC{{\mathbb C}}

\newcommand\CP{{\mathbb C \mathbb P}}

\newcommand\integers{{\mathbb Z}}

\newcommand\naturals{{\mathbb N}}

\theoremstyle{plain}

\theoremstyle{remark}

\renewenvironment{quotation}
{\list{}{
    \setlength\itemindent{0em}%
    \setlength\leftmargin{1.5em}
    \setlength\rightmargin{1.5em}
  }%
\item[]}
{\endlist}

\newcommand\defn[1]{{\bf #1}} 
\newcommand\actson{\circlearrowright}

\newcommand\Gm{{{\mathbb G}_m}}

\newcommand\ZZ{\integers}
\newcommand\NN{\naturals}

\font\co=lcircle10

\def\jr{\smash{\raise2pt\hbox{\co \rlap{\rlap{\char'005} \char'007}}
               \raise6pt\hbox{\rlap{\vrule height6.5pt}}
               \raise2pt\hbox{\rlap{\hskip4pt \vrule height0.4pt depth0pt
                width7.7pt}}}}
\def\je{\smash{\raise2pt\hbox{\co \rlap{\rlap{\char'005}
                \phantom{\char'007}}}\raise6pt\hbox{\rlap{\vrule height6pt}}}}
\def\+{\smash{\lower2pt\hbox{\rlap{\vrule height14pt}}
                \raise2pt\hbox{\rlap{\hskip-3pt \vrule height.4pt depth0pt
                width14.7pt}}}}

\def\textcross{\ \smash{\lower4pt\hbox{\rlap{\hskip4.15pt\vrule height14pt}}
                \raise2.8pt\hbox{\rlap{\hskip-3pt \vrule height.4pt depth0pt
                width14.7pt}}}\hskip12.7pt}
\def\textelbow{\ \hskip.1pt\smash{\raise2.8pt%
                \hbox{\co \hskip 4.15pt\rlap{\rlap{\char'005} \char'007}
                \lower6.8pt\rlap{\vrule height3.5pt}
                \raise3.6pt\rlap{\vrule height3.5pt}}
                \raise2.8pt\hbox{%
                  \rlap{\hskip-7.15pt \vrule height.4pt depth0pt width3.5pt}%
                  \rlap{\hskip4.05pt \vrule height.4pt depth0pt width3.5pt}}}
                \hskip8.7pt}

\usepackage{tikz}\usetikzlibrary{graphs,quotes,fit,positioning,matrix,calc,decorations.markings,angles,decorations.pathmorphing,decorations.pathreplacing}

\tikzset{mynode/.style={circle,draw=black,fill=black,inner sep=1.8pt,outer sep=0pt}}
\tikzset{whitenode/.style={circle,draw=black,fill=white,inner sep=1.8pt,outer sep=0pt}}
\tikzset{edgelabel/.style={\mcol,inner sep=0pt}}
\tikzset{invlabel/.style={draw=black,text=black,circle,inner sep=0pt,minimum size=3mm}}

\begin{document}
\pagestyle{plain}

\title[The Duistermaat-Heckman formula and Chern-Schwartz-MacPherson classes]
{The Duistermaat-Heckman formula and\\ Chern-Schwartz-MacPherson classes}

\author{Allen Knutson}
\email{allenk@math.cornell.edu}
\date{\today}

\dedicatory{For Victor Guillemin, my friend, advisor, and inspiration}

\maketitle

\renewcommand\AA{{\mathbb A}}

\begin{abstract}
  Let $M$ be a smooth complex projective variety, bearing a
  K\"ahler symplectic form $\omega$ and a Hamiltonian action of a
  torus $T$, with finitely many fixed points $M^T$. One standard form of
  the Duistermaat-Heckman theorem gives a formula for $M$'s
  Duistermaat-Heckman measure $DH_T(M,\omega)$ as an alternating sum 
  of projections of cones, with overall direction determined by a
  Morse decomposition of $M$.

  Using Vi{\bf c}tor Ginzburg's construction of Chern-Schwartz-MacPherson
  classes, we show that these individual cone terms can themselves
  be interpreted as
  Duistermaat-Heckman measures of cycles in $T^*M$. (This has a
  similar goal to the symplectic cobordism approach of Vi{\bf k}tor
  Ginzburg, Guillemin, and Karshon.) Our approach also suggests
  extensions of the formula, including the Brianchon-Gram theorem.
\end{abstract}

{\Small
  \setcounter{tocdepth}{2}
  \tableofcontents
}

\section{The Duistermaat-Heckman formula
  (apr\'es \cite{GLS})}\label{sec:DH}

Except for a minor twist, the material in this section is by now
completely classical (no pun intended), and we include it largely to
fix notation. The minor twist will be the inclusion of a cycle in our manifold.
For more leisurely treatments we direct the reader to
\cite{GLS,HK}.

Given
\begin{itemize}
\item a complex projective manifold $M \subseteq \PP V$ where $V$ bears
  a fixed Hermitian form $\langle,\rangle$, hence
\item a symplectic form $\omega$ on $M$, 
  restricted from the Fubini-Study form on $\PP V$,
\item an algebraic cycle $C \subseteq M$, 
  i.e. a formal $\ZZ$-linear combination 
  $\sum_i n_i C_i$ of subvarieties $C_i \subseteq M$ of a fixed dimension, and 
\item an action $T \actson V$ of a compact torus $T$ preserving
  each of $M$, $C$ and $\langle,\rangle$, hence
\item a homomorphism $\rho:T \to U(V)$ and a moment map $\Phi:M \to \lie{t}^*$
  made by composing $M \into \PP V \into \lie{u}(V)^* 
  \xrightarrow{\rho^*} \lie{t}^*$, where $\lie{u}(V)^*,\lie{t}^*$ are the
  duals of the Lie algebras,
\end{itemize}
we have three ways to think about the \defn{Duistermaat-Heckman measure} 
$DH_T(C \subseteq M, \omega)$:
\begin{enumerate}
\item We can use $\omega$ to define a Liouville measure on the smooth part of
  each $C_i$, push those measures forward with $\Phi_*$ to $\lie{t}^*$, 
  and consider the sum of those measures, weighted by the coefficients $(n_i)$.
\item We can take the Fourier transform of $\int_M \exp(\tilde\omega) [C]$,\\
  where $\tilde\omega = \omega - \Phi$ is the equivariant extension of
  the symplectic form (as in \cite{AtiyahBott}). 
  Note that if $C$ is a smooth subvariety (formally
  given coefficient $1$), \\ then
  $\int_M \exp(\tilde\omega) [C] = \int_C \exp(\tilde\omega|_C)$ so
  $DH_T(C\subseteq M,\omega) = DH_T(C\subseteq C,\omega|_C)$.
\item (Following \cite{BrionProcesi} or \cite[\S 3.4]{GLS}) 
  We can compute the weight multiplicity diagrams of
  $\oplus_i \Gamma(C_i; \mathcal O(d))^{n_i}$ on $T$'s weight lattice
  $T^*$, divide the multiplicities by $d^{\dim C}$ and the lattice
  spacing by $d$, and consider the limit of the measure (now on
  $\lie{t}^*$) as $d\to \infty$.
\end{enumerate}

In \cite{AtiyahBott} the first two ways are related (in the case $C=M$, but
our generalization is an easy change) when $M^T$ is finite, 
computing the integral
$$ \int_M \exp(\tilde\omega) [C] 
=
\sum_{f\in M^T} 
\frac{\exp(\tilde\omega)|_f \ [C]|_f} {\prod_{\lambda \in wts(T_f M)} \lambda}
=
\sum_{f\in M^T} \exp(-\Phi(f)) 
\frac{ [C]|_f} {\prod_{\lambda \in wts(T_f M)} \lambda}
$$
where $wts(T_f M)$ denotes the set of weights (with multiplicity) in the
isotropy action of $T$ on the tangent space $T_f M$,
and $\alpha|_f$ denotes the pullback of a class $\alpha$ along the 
($T$-equivariant) inclusion of the point $f$.

For general $C$ of complex dimension $d$, we can (nonuniquely) write 
the restriction $[C]|_f$ of its equivariant cohomology class $[C]\in H^*_T(M)$
to the point $f$ as
$$  [C]|_f = \sum_{S \in {wts(T_f M) \choose d}} n_{f,S} 
\prod_{\lambda \in wts(T_f M)\setminus S} \lambda,\qquad \text{for some }n_{f,S} \in \ZZ 
$$ 
where $S$ runs over sub-multisets of $wts(T_f M)$ of size $d$. 
(Proof: pass to the formal neighborhood of $f$ in $M$, then apply
\cite[Lemma D]{grobGeom}, which is stated there over polynomial rings but 
applies without change to power series rings.) 

It becomes now very tempting to Fourier transform the sum term by term. 
As is well-understood, 
defining the transform of such singular terms requires some choices
in regularization, which we recapitulate in a moment. For now, we choose
$\vec v \in \lie t$ not perpendicular to any of the (finitely many)
weights in $\{T_f M\colon f\in M^T\}$, 
use it to define $ \lambda_+ := \sign(\langle \vec v,\lambda \rangle) \lambda $,
and flip the signs on some of our $(n_{f,S})$ so that
$$  \frac{[C]|_f}{\prod_{\lambda \in wts(T_f M)} \lambda}
= \sum_{S \in {wts(T_f M) \choose d}} n_{f,S}\ \frac{1}
{\prod_{\lambda \in S} \lambda_+}
$$ 
and having now rewritten the ``equivariant multiplicity'' of $[C]$
like so, we have
\begin{eqnarray*}
  \int_M \exp(\tilde\omega) [C] 
\junk{  &=& \sum_{f\in M^T} \exp(-\Phi(f)) 
      \sum_{S \in {wts(T_f M) \choose d}} n_{f,S}\ \frac{1}
{\prod_{\lambda \in S} \lambda_+} \\}
  &=& \sum_{f\in M^T,\ S \in {wts(T_f M) \choose d}} n_{f,S}\
      \frac{\exp(-\Phi(f))}{\prod_{\lambda \in S} \lambda_+}
\end{eqnarray*}

\subsection{Cone terms}

Define a \defn{cone term} associated to a weight $\mu \in T^*$
and a multiset $P = \{\lambda\}$ of weights in $T^*$
as any multiple of
$$ cone\left(\mu, \{\lambda\}\right) 
:= \pi_*(\text{Lebesgue measure on }{\RR_{\geq 0}}^P),
$$
where
\junk{$$
\begin{array}{rrcl}
  \pi:& (\RR\tensor L) \times{\RR_{\geq 0}}^P &\to& \lie{t}^* \\
  &(\vec v, (x_\lambda)_{\lambda\in P}) 
      &\mapsto&\mu + \vec v + \sum_{\lambda\in P} x_\lambda \lambda
\end{array}
$$
and we normalize the measure on $(\RR\tensor L) \leq \lie{t}^*$ to make
$vol((\RR\tensor L)/L) = 1$. 
The term doesn't fully depend on $P$,
just its image in $T^*/L$. If $L=0$ call the cone term \defn{pointed},
and if $\dim L = 1$ call it \defn{$1$-lineal}.
}
$$
\begin{array}{rrcl}
  \pi:& {\RR_{\geq 0}}^P &\to& \lie{t}^* \\
  &(x_\lambda)_{\lambda\in P}
      &\mapsto&\mu  + \sum_{\lambda\in P} x_\lambda \lambda
\end{array}
$$

This measure is only locally finite on $\lie{t}^*$ if $\pi$ is proper,
or equivalently, if the vectors $\lambda \in P$ 
live in an open half-space in $T^*$. 
This, we will arrange by requiring $\langle \vec v,\lambda\rangle > 0$.
At this point we {\em define} the Fourier transform of
${\exp(-\Phi(f))}/{\prod_{\lambda \in S} \lambda_+}$ to be $cone(\Phi(f),S)$.
Summing these cone terms, we have arrived at
\defn{Heckman's formula} (as it is called in \cite[\S3.3]{GLS}) for
the Duistermaat-Heckman measure.

\subsection*{Example: $\CP^2$} Let $T^2 \actson \CC^3$ with weights
$P = \{(0,0),(1,0),(0,1)\}$, and let $\vec v = (1,2)$. Let $M$ and $C$ be the
projectivization $\CP^2$. Then $DH_T(\CP^2, \omega)$ is Lebesgue measure
on the triangle with vertices $P$. The formula above computes it as pictured,
where cyan edges indicate flipped edges, 
i.e. those for which $\lambda_+ = -\lambda$.
(The dotted lines, which indicate a certain $3$-dimensionality of
the picture, will be explained later.)

\centerline{ \epsfig{file=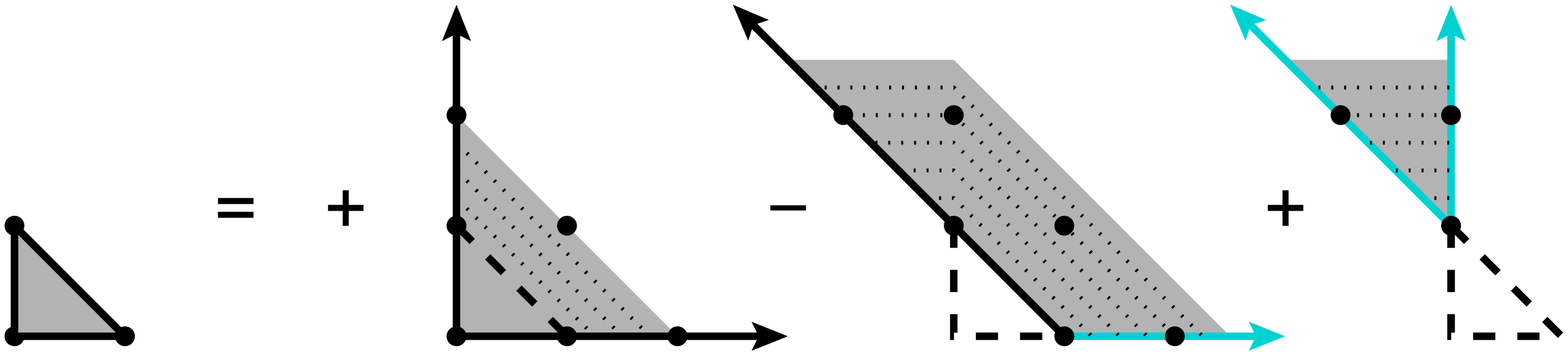,width=6in}} 

\subsection{Fourier equivalent measures}
There is a clear subtlety -- if we change $\vec v$, it changes the
individual terms, so what precisely is guaranteeing that the final sum
of measures is independent of $\vec v$? (Other than the obvious: the measure
was originally defined as a pushforward not depending on this choice.)

In what is to come the only measures we need consider on $\lie{t}^*$ will
be linear combinations of cone terms. Call two cone terms
$C\cdot cone(\mu,{\lambda})$ and $D\cdot cone(\mu',{\lambda'})$
\defn{Fourier equivalent} if $\mu = \mu'$, the two multisets 
$\{\lambda\}, \{\lambda'\}$ agree up to negating $k$ many weights,
and $C = (-1)^k D$. Call two measures Fourier equivalent if one can
be obtained from the other by replacing terms in one 
with Fourier equivalent terms to obtain the other. We thank Terry Tao
for pointing out the following lemma: 

\begin{Lemma}
  If $f,g$ are Fourier equivalent linear combinations of cone terms, and 
  there is a pointy cone $P$ such that $f,g$ are both supported with $P$,
  then $f=g$.
\end{Lemma}

\begin{proof}
  \junk{  Let $f = \sum_i C_i cone(\mu_i, W_i)$,
    $g = \sum_i D_i cone(\mu'_i, W'_i)$.
    
    Since $P$ is a pointy cone, there is a functional $X \in \lie{t}$ 
    on $\lie{t}^*$ such that the map $X:\lie{t}^* \to \RR$ 
    is proper on $P$, and thereby, on the supports of $f$ and $g$.
    We can also pick $X$ generically enough that it takes on different
    values on distinct $\mu_i,\mu'_i$.
  }
  Let $\partial_\lambda$ denote the differencing operation 
  $(\partial_\lambda f)(x) = f(x) - f(x+\lambda)$. 
  If $\partial_\nu f = \partial_\nu g$ with $\nu\neq \vec 0$, then $f$
  and $g$ differ by a function $h$ invariant under translation by $\nu$.
  This $h$ will also be supported inside $P$, hence $h=0$. So we can
  simplify the equality to be checked by applying such operators,
  without losing information.

  If we take any single cone term $cone(\mu,\{\lambda\})$ and
  apply $\prod_\lambda \partial_{\lambda}$, the result is the projection of
  a parallelepiped. 
  At this point the two measures we are comparing are compactly supported,
  and are thus determined by their Fourier transforms.
\end{proof}

Heckman's formula produces a measure supported in a translate of the
cone spanned by $\{\lambda_+\colon \lambda$ a weight in some $T_fM\}$, 
which is a proper cone (i.e. pointy) by the assumption on $\vec v$.
Since the manifold $M$ is compact its DH measure is compactly supported, 
hence is supported in some translate of any proper cone. 
At this point we apply the lemma.

In what is to come we will need noncompact extensions of (1-3), as
to be found in \cite{PW, Thomason}. Instead of compactness of $M$, 
one reduces to the case $M$ connected, 
then asks that $M^T$ be compact, and finally that some component of $M^T$ be
\defn{attractive}, meaning that all of the isotropy weights in its
normal bundle lie in an open half-space of $T^*$. The moment map is then proper,
so (1) makes sense, and its image lies inside a proper cone, as was needed
in the argument above. One uses the AB/BV localization theorem
to make sense of (2), and although $\Gamma(C_i;\calO(d))$ is 
infinite-dimensional its weight spaces are finite-dimensional, 
making sense of (3).

Finally, we will need to relax the nondegeneracy of the symplectic form,
to a closed $2$-form. This changes (1) in that the moment map is no
longer determined up to translation (unless the form is generically
nondegenerate) but is explicit extra data.
It does not affect definition (2). 
For definition (3) one still needs $[\omega]$ to be the first Chern class
of a holomorphic line bundle $\calO(1)$, and the weight multiplicity
diagram is now made using the entire Euler characteristic of the sheaf
cohomology of $\calO(d)$, rather than just $H^0$ the space of sections.  
In the rest of the paper we work primarily with definition (2), based off
the AB/BV localization formula.

\section{Chern-Schwartz-MacPherson classes}\label{sec:csm}

We follow \cite{Ginzburg} for our treatment of CSM classes, as derived
from $\calD$-modules.

\subsection{An exact sequence of $\calD$-modules}\label{ssec:exact}

Let $\beta: B \into A$ be the closed inclusion of one smooth
complex manifold into another, in codimension $1$,
defined by the vanishing of a function $f$.
Then there is a short exact sequence of $\calD_A$-modules
$$ 0 \to \calO_A \xrightarrow{\cdot f}
\calO_{A\setminus B} \to \beta_*(\calO_B) \to 0 $$
where $\calD_A$ is the sheaf of differential operators on $A$,
and $\beta_*$ is the pushforward of $\calD$-modules. 

Our running example is very simple: $\{0\} \into \CC$. 
Since $\CC$ is affine, instead of working with sheaves we can take
global sections $\Gamma(\calD_\CC) \iso \CC[\hat x,\frac{d}{dx}]$.
Then the short exact sequence of $\CC[\hat x,\frac{d}{dx}]$-modules
$$
\begin{array}{rccccccl}
  0&\to&
        \CC[\hat x,\frac{d}{dx}]/\langle \frac{d}{dx} \rangle &\into& 
        \CC[\hat x,\frac{d}{dx}]/\langle \frac{d}{dx} \hat x \rangle &\onto& 
        \CC[\hat x,\frac{d}{dx}]/\langle \hat x \rangle &\to 0 \\ \\
  && \hfill 1&\mapsto&\hat x \hfill 1&\mapsto& 1 \hfill \\ \\
  \text{define ODEs} && \frac{d}{dx}f = 0 
       && \left(\frac{d}{dx}\hat x\right) f = 0 && \hat x f = 0 \\ \\
 \text{with solutions} && 1 && x^{-1} && \delta &\text{($\delta$ = Dirac delta)}
\end{array}
$$
and those solutions can be identified with generators 
of our $\calD_\CC$-modules:
$$
\begin{array}{rcccccccc}
  0&\to& \CC[x] &\into& x^{-1}\CC[x^\pm] &\onto& \delta \CC[\delta] &\to &0
\end{array}
$$

\subsection{Characteristic cycles}

The $\calO_A$-algebra $\calD_A$ is generated by vector fields, 
sections of $TA$, which are used to build directional derivatives and therefore
define (noncommuting) operators on $\calO_A$. 
If we instead interpret sections of $TA$ as fiberwise linear functions
{\em on} $T^*A$, then they generate a different, commutative, algebra:
the sheaf of (polynomial, not just linear) functions on $T^* A$.
One can make the relation more precise: the degree of differential operators 
induces a filtration on $\calD_A$, whose associated graded algebra
$gr\ \calD_A$ is $\calO_{T^* A}$. 

Given a $\calD_A$-module $\calF$, one could hope to filter it as well,
compatibly with the $\calD_A$ filtration. At that point 
$\calO_{T^* A} \actson gr\ \calF$, and we can consider its support cycle
$supp(gr\ \calF) \subseteq \calO_{T^* A}$ (which will typically have
multiplicities). It turns out (see e.g. \cite[definition 1.8.5]{Bjork})
that such ``good filtrations'' exist, not uniquely enough to canonically 
define the sheaf $gr\ \calF$, but uniquely enough to well-define its 
support cycle. In the running example from \S\ref{ssec:exact}, we get the 
following short exact sequence of modules over $gr\ \calD_\CC \iso \CC[x,y]$:
$$
\begin{array}{rcccccccl}
   0 
&  \to& \CC[x,y]/\langle y\rangle 
&\into& \CC[x,y]/\langle xy\rangle 
&\onto& \CC[x,y]/\langle y\rangle 
&  \to& 0 \\
  &&\hfill 1&\mapsto&x\hfill 1&\mapsto&1\hfill &
\end{array}
$$
The gradedness of $gr\ \calF$ can be interpreted as its bearing a circle action.
This has the specific consequence that $supp(gr\ \calF)$ is a 
\defn{conical} cycle inside $T^*A$, meaning, invariant under the 
dilation action $\CC^\times \actson T^*A$ that scales the cotangent vectors.
Consequently, we get a well-defined class
$$
[supp(gr\ \calF)] 
\ \in\ H^*_{\CC^\times}(T^*A) \quad
\iso H^*_{\CC^\times}(A) 
\iso H^*(A) \tensor H^*_{\CC^\times}(pt)
\iso H^*(A)[\hbar]
$$
in the dilation-equivariant cohomology of $T^*A$,
taking $\hbar$ as the generator of $H^*_{\CC^\times}(pt)$.

Perhaps the simplest example is $\calF = \calO_A$. 
Then $supp(gr\ \calF)$ is the zero section $A \subseteq T^*A$, and
its associated class is a $(-\hbar)$-homogenized version of the
total Chern class of the tangent bundle of $A$.

In this paper the only (complexes of) $D_M$-modules we will need consider 
are of the form $R\iota_*(\calO_A)$ for $\iota:A\into M$ the inclusion
of a locally closed submanifold. Hereafter we write
$$ cc(A\subseteq M) := \sum_i (-1)^i\ supp\left(gr\ R^i\iota_*(\calO_A)\right) $$
to denote the resulting ``characteristic cycle'', 
a conical Lagrangian cycle inside $T^*M$ defining an element
$[cc(A\subseteq M)] \in H^*_{\CC^\times}(T^*M) \iso H^*(M)[\hbar]$.
In the running example the derived pushforward $R^i\iota_*$ vanishes
for $i>0$, but in other examples such as $\CC^2\setminus 0 \into \CC^2$
(described in more detail later)
one must include some such higher pushforwards.

In general $cc(A\subseteq M)$ is very complicated, with many components
with various multiplicities. One component is the closure of the 
conormal bundle to $A$, and each other component is the closure of the
conormal bundle to some locally closed submanifold 
$B \subseteq \overline A \setminus A$.

\subsection{Chern-Schwartz-MacPherson classes 
  and their additivity}\label{ssec:csm}

\newcommand\calA{{\mathcal A}}

Recall that a \defn{constructible function} $f$ on $M$ is a finite linear
combination (with $\ZZ$-coefficients, say) of characteristic functions
of closed subvarieties. By splitting a subvariety into its regular and
singular locus, we can instead think of $f$ as a linear combination
$\sum_{A\in \calA} n_A 1_A$ of characteristic functions of locally
closed algebraic submanifolds $A\subseteq M$. This expansion is not
unique, however, so we need to treat it with care.

For example, write $1_\CC = 1_{\CC^\times} + 1_{0}$. To these three subsets
we can associate $\calD_\CC$-modules as computed in our running example,
and the characteristic cycles $\{y=0\}$, $\{xy=0\}$, $\{x=0\}$ respectively.
The exactness of the sequence from \S\ref{ssec:exact} leads to a
vanishing of its Euler characteristic, as an equation on cycles:
$$
\begin{array}{rccccccc}
&  cc(\CC \subseteq \CC)
  &-& cc(\CC^\times \subset \CC) &+& cc( \{0\}\subset \CC) &= &0 \\ \\
\text{Pictorially:} \quad&
  \text{\huge --} &\text{minus}& \text{\huge +} &\text{plus} & \big |
& =&0
\end{array}
$$
This alternating-sum statement doesn't quite match $1_\CC-1_{\CC^\times}-1_{0} = 0$.
To fix this mismatch, for $\iota_*: A \subseteq M$ a locally closed 
submanifold, we define its \defn{Chern-Schwartz-MacPherson class} $csm(1_A)$ as
$$ csm(1_A) := (-1)^{\codim_M A}\ [cc(A \subseteq M)]\qquad\in H^*_{\CC^\times}(T^*M).
$$
With the signs integrated into the definition, it is then a theorem that
this definition on $\{1_A\}$ extends in a well-defined way to 
constructible functions, at which point it is additive.
(While it wasn't important in the running example, in bigger examples
this additivity relies on $cc$ having been defined using the derived 
pushforward.)

The traditional definition is slightly off from this -- it lives in homology 
rather than cohomology, and is dehomogenized by setting $-\hbar$ to $1$.
We take this opportunity to rant about the horrific unnaturality of
considering inhomogeneous elements of cohomology, insofar as cohomology
should so very often be understood as the associated graded to $K$-theory.
In an associated graded space, only homogeneous elements can properly be 
asked to possess lifts. While this concludes the rant, we will retain
the powers of $\hbar$ through the rest of this paper.

The Deligne-Grothendieck conjecture, proven by MacPherson \cite{MacPherson}, 
characterized these CSM classes by a recurrence relation (a functoriality under
proper pushforward) and a base case ($A=M$ smooth and proper). 
One philosophical reason to prefer the description from
\cite{Ginzburg} recapitulated here is its individual definition for each
$A \subseteq M$, rather than reliance on a recurrence relation.

\subsection{Weber's divisibility property}

Taking $A\subseteq M$ as before, and a torus $T$ acting on $M$ preserving $A$,
then we can use the same definition to associate an 
\defn{equivariant CSM class} $csm(1_A) \in H^*_{T\times \CC^\times}(T^* M)$
(a slightly different approach appears in \cite{Ohmoto}).

\begin{Lemma}\cite[theorem 20]{Weber} \label{lem:weber} 
  If $p \in M^T$ is isolated, and $p\notin A$, 
  then $csm(1_A)|_p \equiv 0 \bmod \hbar$.
\end{Lemma}

That requires a straightforward bit of translation from \cite{Weber},
as CSM classes there are inhomogeneous.

\junk{

\begin{Lemma}
  Let $T\actson M$, and let $L \subseteq T^*M$ be a $T$-invariant
  conical Lagrangian subvariety. Let $p\in M^T$ be an isolated fixed point,
  with isotropy $T$-weights $P$ (a set with multiplicity).
  Let $Power(P)$ be the set of all sub-multisets of $P$
  (i.e. if $\lambda$ occurs in $P$ with multiplicity $k$, and $S\in Power(P)$,
  then $\lambda$ occurs in $S$ with some multiplicity $k'\leq k$,
  and in $P\setminus S$ with multiplicity $k-k'$).

  Then there exists a function $n\colon Power(P)\to \NN$ such that
  $$ [L]|_p = \sum_{S \in Power(P)} n(S) 
  \left(\prod_{\lambda \in S} \lambda\right)
  \left(\prod_{\lambda \in P\setminus S} (\hbar-\lambda) \right)
  $$  
\end{Lemma}

\begin{proof}
  We first pass to
  the formal neighborhood $M_{(p)}$ of $p\in M$, and its cotangent bundle,
  and to the germ $L_{(p)}' \subseteq T^*(M_{(p)})$; it now suffices to
  compute $[L_{(p)}]|_p \in H^*_{\CC^\times \times T}(T^*(M_{(p)}))$.
  The remainder of the proof will be a sort of cotangent bundle
  version of \cite[lemma D]{grobGeom}. 

  Pick co\"ordinates $M_{(p)} \iso \CC[[x_1,\ldots,x_n]]$. To lex-degenerate 
  a particular co\"ordinate $x_i$, we consider the ring automorphism scaling
  $x_i$ by a parameter $t$, and let $t\to\infty$; if we do this to each of
  the co\"ordinates in order we 
\end{proof}

}

\section{The main theorem:
  the geometry of cone terms}\label{sec:cones}

We are now ready to give an algebro-geometric interpretation of the individual 
cone terms in Heckman's formula.
This is in similar spirit to the approach
of \cite{GGK}, where these cone terms are interpreted as components of
the (other end of the) boundary of a noncompact symplectic cobordism.
It would be interesting to connect the two approaches, perhaps through the
algebraic cobordism of \cite{LP}.

Our input is a complex projective symplectic manifold 
$(M,\omega)$ with an algebraic action of a torus $T$, 
and a Bia{\fontfamily{ppk}\selectfont \l}ynicki-Birula 
decomposition defined using a circle $S\into T$
(or equivalently, a Morse decomposition defined using a component of
$T$'s moment map). We assume $M^T$ finite, and $S$ generic enough that
$M^S = M^T$, then write the decomposition
into attracting sets as $M = \coprod_{p\in M^T} M_p^\circ$. 
Each inclusion $\iota^p: M_p^\circ \into M$ defines a Lagrangian cycle
$cc(M_p^\circ \subseteq M)$ in $T^*M$, as in \S\ref{sec:csm}.

\newcommand{\xrightarrowdbl}[2][]{%
  \xrightarrow[#1]{#2}\mathrel{\mkern-14mu}\rightarrow
}

We give $T^*M$ the degenerate $2$-form $\omega_+ := \pi^*(\omega)$
where 
$M \xhookrightarrow\iota T^*M\xrightarrowdbl\pi M$ 
are the inclusion and projection. 
This choice is dictated by wanting
$\omega_+$ to be {\em invariant,} not just a weight vector, 
under the dilation action $\CC^\times \actson T^*M$ on the fibers, 
and wanting $\iota^*(\omega_+)$ to be $\omega$. (There is a familiar
nondegenerate $2$-form ``$d\alpha$'' available on $T^*M$, 
which one might be tempted to add to $\omega_+$.
As that form is exact it wouldn't affect our cohomology-based calculations,
but it {\em would} spoil the dilation-invariance.) Now we compute:
$$
\begin{array}{rcll}
  && DH_{\CC^\times\times T}(M \subseteq M, \omega) \\[.1cm]
  &=& DH_{\CC^\times\times T}(M \subseteq T^*M, \omega_+) 
  &\text{since $\iota^*(\omega_+) = \omega$} \\[.2cm]
  &=& F.T. \int_{T^*M} [M] \exp(\tilde\omega_+) 
  \ =\ F.T. \int_{T^*M} csm(1_M) \exp(\tilde\omega_+) \\[.1cm]
  &=& F.T. \int_{T^*M} csm\left(\sum_p 1_{M_p^\circ}\right) \exp(\tilde\omega_+) \\[.2cm]
  &=& \sum_p F.T.\int_{T^*M} csm\left(1_{M_p^\circ}\right) \exp(\tilde\omega_+) 
      &\text{by additivity of CSM classes}\\[.2cm]
  &=& \sum_p F.T. \int_{T^*M} (-1)^{\codim_M M_p^\circ}
      [cc(M_p^\circ \subseteq M)] \exp(\tilde\omega_+) 
  & \text{as in \S\ref{ssec:csm}}\\[.2cm]
  &=& \sum_p (-1)^{\codim_M M_p^\circ}\ F.T. \int_{T^*M} 
      [cc(M_p^\circ \subseteq M)] \exp(\tilde\omega_+) \\[.2cm]
  &=& \sum_p (-1)^{\codim_M M_p^\circ}\ 
      DH_{\CC^\times\times T}(cc(M_p^\circ \subseteq M) 
      \subseteq T^* M, \omega_+) 
\end{array}
$$
Before comparing (not equating!) our terms 
$DH_{\CC^\times\times T}(cc(M_p^\circ \subseteq M) \subseteq T^* M, \omega_+)$
to the cone terms in the Heckman formula, we consider 
the basic example $M = \CP^1 = \CC \coprod \{\infty\}$. 
Pictured below are the moment images of $T^*M$, 
$cc(\CC\subseteq \CP^1)$, $cc(\{0\} \subseteq \CP^1)$
all with respect to the two-torus $\CC^\times \times T$, 
where $T$ is the maximal torus of $PGL_2(\CC)$ and $\CC^\times$ acts by
dilation on the cotangent fibers. 

\epsfig{file=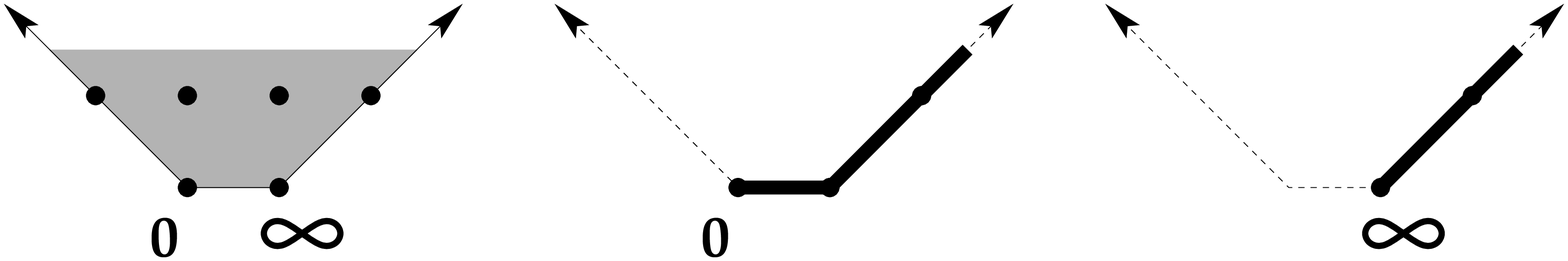,width=6in}

Note that in this tiny example, $T^*M$ is toric w.r.t. our augmented
torus $\CC^\times \times T$, but this will never happen in larger examples.

Even without the pictures, there is an obvious difference between the
terms in this alternating sum vs. the ones in the Heckman formula: 
in {\em this} sum, the terms involve an extra $\CC^\times$ action,
dilating the fibers of the cotangent bundle. (Note too that the moment
image of $cc(\CC\subseteq \CC\PP^1)$ is not a polytope, which can be
blamed on the characteristic cycle being reducible.)

To drop that action, consider the inclusion $T \into S^1 \times T$
inducing $T^* \times \ZZ \onto T^*$, which
\begin{itemize}
\item on the cohomology algebra level, amounts to setting $\hbar\to 0$, and
\item on the moment polytope level, amounts to pushing forward the measure
  along the \break projection $\lie{t}^* \times \RR \onto \lie{t}^*$.
\end{itemize}

\noindent
\begin{minipage}{.45\linewidth}
    \parindent = 1em
  In the $2$-dimensional pictures above, that amounts to projecting
  the measures to the horizontal line, from which to obtain the usual
  formula for $DH_T(\CP^1,\omega)$:
\end{minipage}
\hfill
\begin{minipage}{.45\linewidth}
 \epsfig{file=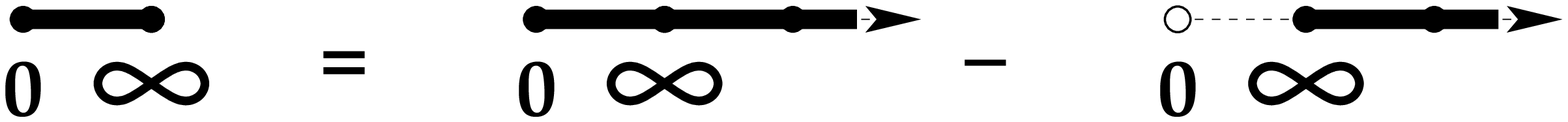,width=3in}
\end{minipage}
\hfill \phantom{.}

\noindent
One might consider those $2$-dimensional pictures above as ``bent''
versions of the half-lines that we really want, but that we only obtain after
the dilation action is suppressed. Looking back at the $\CP^2$ example in
\S\ref{sec:DH}, this extra action can be pictured by seeing the
triangle as flat in the page, and the other regions as coming out of
the page. The dotted lines indicate level sets in those
$3$-dimensional pictures.

We have now arrived at the main theorem, giving geometric interpretation
to the individual terms in the Heckman formula:
it is almost correct (and conjecturally correct) to say 
they are {\em themselves} DH measures, not of $M$ 
but of the characteristic cycles $cc(M_p^\circ \subseteq M) \subseteq T^*M$.

\begin{Theorem}\label{thm:main}
  Let $S\into T \actson (M,\omega)$, $M = \coprod_{p\in M^T} M_p^\circ$ be as
  described at the beginning of \S\ref{sec:cones}.
  Then $DH_{\CC^\times\times T}(cc(M_p^\circ \subseteq M)\subseteq T^*M, \omega_+)$ 
  is Fourier equivalent (and conjecturally equal)
  to a measure whose projection to $\lie{t}^*$ is proper, 
  and that projection is $p$'s cone term from the Heckman formula.
\end{Theorem}

There are two subtleties in the theorem's statement. What we are really after
could reasonably be called $DH_T(cc(M_p^\circ \subseteq M) \subseteq T^*M,
\omega_+)$. One problem is that the $T$-moment map on $T^*M$ isn't proper
for $\dim M>0$. We {\em believe} that its restriction to
$cc(M_p^\circ \subseteq M)$ {\em is} proper, but (a) this has been 
frustratingly elusive and (b) that characteristic cycle is typically singular 
so we prefer to keep our integration definition on the manifold $T^*M$.
(In \S\ref{ssec:problematic} we study a slightly different situation 
where the $T$-moment map is {\em not} proper on the characteristic cycle.)

\begin{proof}
  Let $C$ denote the cycle $cc(M^\circ_p \subseteq M)$. 
  We recall that it consists of the closure of the conormal bundle to 
  ${M^\circ_p}$ union various other conormal varieties living
  over $\overline{M^\circ_p} \setminus M^\circ_p$. In particular
  $$ [C]|_p = \prod_{\lambda \in T_p M}
  \begin{cases}
    \lambda & \text{if $\lambda$ defines a positive $S$-weight} \\
    \hbar-\lambda & \text{if $\lambda$ defines a negative $S$-weight} \\
  \end{cases}
  $$
  (no $\lambda$ will define the $S$-weight $0$, by our choice of $S$).

  The Fourier transform of $C$'s DH measure on $\lie t^* \times \RR$
  is $\sum_{f\in M^T} \exp(-\Phi(f)) 
  \frac{ [C]|_f} {\prod_{\lambda \in wts(T_f M)} \lambda(\hbar-\lambda)}$.
  To associate a measure to it (which might only be Fourier equivalent
  to the actual DH measure), as in \S\ref{sec:DH} we need to flip some 
  weights in denominators. We make that choice using the 
  generator of $S$'s Lie algebra. It is easy to see that the 
  resulting measure has proper projection along the composite
  $\lie{t}^* \times \RR \onto \lie{t}^* \onto \lie{s}^*$, hence 
  has proper projection to $\lie{t}^*$.

  To compute the projection to $\lie{t}^*$, on the Fourier transform side,
  amounts to setting $\hbar\to 0$. Now we use lemma \ref{lem:weber}
  to note that $[C]|_f \equiv 0 \bmod \hbar$ for $f\neq p$. Hence our
  sum reduces to a single term
  $$ \left( \exp(-\Phi(p)) 
    \frac{ [C]|_p} {\prod_{\lambda \in wts(T_p M)} \lambda(\hbar-\lambda)}
  \right)\bigg|_{\hbar\to 0} = (-1)^{\codim_M M_p^\circ} \exp(-\Phi(p))
  \bigg / {\prod_{\lambda \in wts(T_p M)}} \lambda_+ $$
  which is exactly the term in the localization formula.
\end{proof}

It is worth spelling out the interconnectedness of the different
points of view in the case of the flag manifold, as in Heckman's
thesis, which gives the asymptotic version of Kostant's multiplicity
formula (see \cite[\S 3]{GLS} for the connection).
Our derivation is based on the $\calD_{G/B}$-modules
associated to Bruhat cells; the global sections of these are the Verma modules.
The exact sequence given in \S\ref{ssec:exact} for a single divisor,
when extended to the full Bruhat decomposition, gives the BGG resolution
involving those Verma modules (see \cite{Kempf}).
The complexity we meet here, with the ``bending'' of the individual
cone terms, is closely related to the complexity (the non-simplicity)
of Verma modules. (It is not {\em quite} the same complexity, as even
a simple $\calD$-module can have reducible characteristic cycle, a
well-known example being that of Kashiwara-Saito.)

\section{The Brianchon-Gram theorem 
  and other extensions}\label{sec:bg}

Let $M$ be a smooth projective toric variety, with a moment polytope 
$P \subset \lie{t}^*$. Instead of using a Morse decomposition, we consider
the full decomposition $M = \Union_{F\subset P} M_F^\circ$ into $T^\CC$-orbits,
one for each face of $P$. Then as in \S\ref{sec:csm} we obtain
\begin{equation}
  \label{eq:BG}
    DH_{\CC^\times\times T}(M \subseteq M, \omega) 
  =  \sum_F (-1)^{\codim_P F}\ 
  DH_{\CC^\times\times T}(cc(M_F^\circ \subseteq M) 
  \subseteq T^* M, \ \omega_+) 
\end{equation}
(though the $\codim$ in the exponent is now the {\em real} codimension).
Then as in theorem \ref{thm:main}, we project the measure from
$\RR\times \lie{t}^*$ to $\lie{t}^*$. 

\begin{Theorem}\label{thm:BG}
  Let $(M,\omega)$ is a symplectic toric manifold with moment
  polytope $P\subseteq \lie{t}^*$, let $M_F^\circ \subseteq M$
  be the $T$-orbit corresponding to a face $F\subseteq P$,
  and let $v$ be any vertex of $F$.
  Let $T_\mu F$ denote the (real) tangent space to an interior point
  $\mu$ of $F$, and $W \subset T^*$ denote the primitive integer
  vectors along the edges from $v$ out of $F$.

  Then the pushforward to $\lie{t}^*$ of the measure
  $DH_{\CC^\times\times T}(cc(M_F^\circ \subseteq M) \subseteq T^* M, \ \omega_+)$
  is
  $$ \pi_*(\text{Lebesgue measure on $T_\mu F \times {\RR_{\geq 0}}^W$}),
  \quad\text{where}\quad
  \begin{array}{rrcl}
    \pi:& {T_\mu F \times \RR_{\geq 0}}^W &\to& \lie{t}^* \\
        &(\vec v, (x_\lambda)_{\lambda\in W})
      &\mapsto&\mu + \left(\vec v - \sum_{\lambda\in W} x_\lambda\lambda\right)
  \end{array}
  $$
  where we normalize the measure on $T_\mu F \leq \lie{t}^*$ using its
  intersection with the lattice $T^*$.

  If we pushforward the LHS of equation (\ref{eq:BG}) to $\lie{t}^*$,
  we get Lebesgue measure on $P$, and the pushforward of the RHS
  gives the ``Brianchon-Gram formula'': an alternating sum over all $F$ of
  the cone centered at $F$, with lineality space $T_\mu F$, and 
  generators $W$ as defined above.
\end{Theorem}

The Brianchon-Gram formula was given a Heckman-like derivation also in
\cite{HK}, through a somewhat technical construction of a function
with one critical point in the interior of each face.

\subsection*{Revisiting $\CP^2$}
Before getting into the proof of theorem \ref{thm:BG},
we look again at the $\CP^2$ example from \S\ref{sec:DH}.
There are seven orbits, where the open orbit
gives the entire plane, the $1$-dimensional orbits give half-planes,
and the fixed points give sectors. (As before, we have attempted to indicate
the $\hbar$ direction out of the plane, using dashed level sets. The first
characteristic cycle has seven components and the next three each have three.)
We exhort the reader to check that the choices of vertices $v \in F$
are immaterial.

\epsfig{file=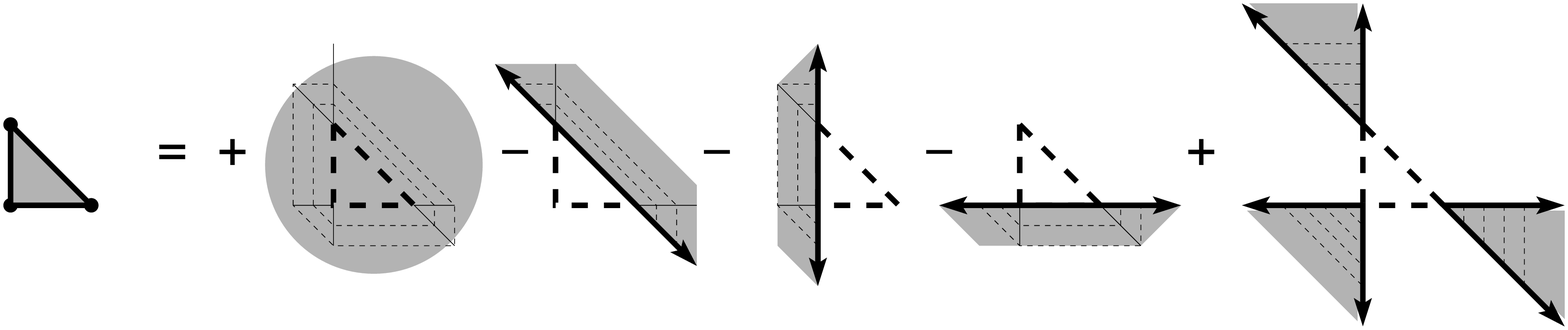,width=6in}

\begin{proof}[Proof of theorem \ref{thm:BG}] 
  For this we use the decomposition of $M$ into its $T^\CC$-orbits.
  For each such orbit $E^\circ \subset M$, with closure we call $E$, 
  observe that $1_{E^\circ} = \sum_{F \subseteq E} (-1)^{\codim_E F}\ 1_F$
  where the sum is over smaller $T^\CC$-orbit closures. Then
  $$ csm(1_{E^\circ}) = \sum_{F \subseteq E} (-1)^{\codim_E F} csm(1_F) $$
  and since $F$ is smooth, and closed in $M$, its characteristic cycle
  is just its conormal bundle $C_M F$. Consequently
  $ [cc(1_{E^\circ})] = \sum_{F \subseteq E} [C_M F] $
  (where the sign we had from inclusion-exclusion cancels with the
  one in Ginzburg's formula for CSM classes). 

  To understand the DH measure associated to this sum, consider the
  (non-central) hyperplane arrangement defining the polytope $P$,
  and many other regions in $\lie{t}^*$. Not every region touches $P$
  (unless $P$ is a product of simplices -- for a first cautionary example,
  consider a trapezoid), so we work in a small open neighborhood $P_+$
  of $P$ to avoid consideration of those other regions. Each hyperplane 
  divides   space into an ``inside'' (where $P$ is) and an ``outside''.
  The moment polytopes of the individual $C_M F$ in the sum, intersected
  with $P_+$, are the exactly the regions that touch $F$ and are on 
  the outside of each of the hyperplanes through $F$. When we add them,
  we get the Brianchon-Gram term associated to $F$.  
\end{proof}

There is another theorem also called Brianchon-Gram, in which the cones
point inward rather than outward, but the total is $(-1)^{\dim P}$
times Lebesgue measure on $P$. One can obtain that from this
by scaling the symplectic form on $M$ by $-1$, which turns $P$ inside out.

The additivity of CSM classes suggests that we wildly generalize to {\em any}
$T$-invariant decomposition of $M$ into locally closed submanifolds. 
We give an example now to demonstrate the dangers.

\subsection{A problematic decomposition}\label{ssec:problematic}
Let $T$ be one-dimensional this time, acting on $V$ with weights $0,1,2$,
and decompose $\PP V$ into the projective point $[0,*,0]$ and the 
open complement $A$. Then $DH_{T \times \CC^\times}(\CP^2,\omega_+)$ 
is a piecewise-linear function times Lebesgue measure on the interval
connecting $(0,0)$ and $(2,0)$. The inclusion $\iota:A\into \CP^2$ is 
perhaps already worrisome in that $R^1\iota_* \neq 0$. Forging ahead, 
we calculate
$DH_{T\times \CC^\times}(cc(S\subseteq M)\subseteq T^*M, \omega_+)$
for $S \in \{ \{[0,*,0]\}, A \}$ and obtain the following picture: \\
\centerline{\epsfig{file=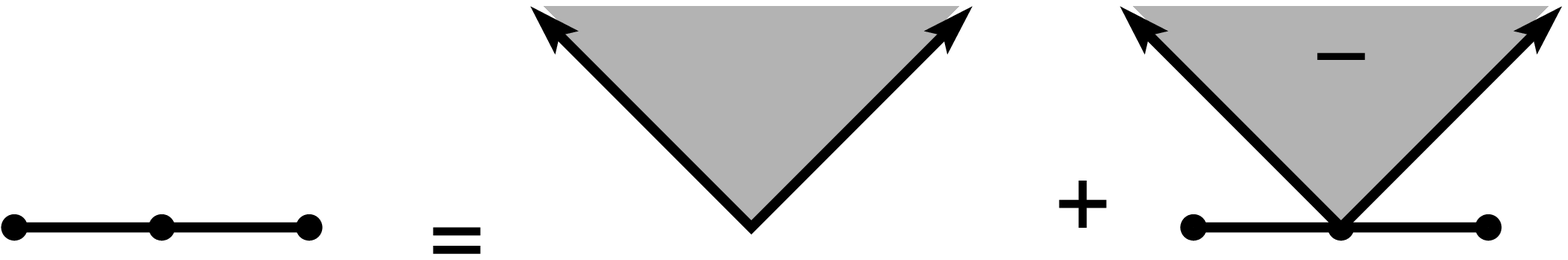,height=.7in}} \\
In the shaded regions we have $\pm\frac{1}{2}$ Lebesgue measure
(where the ``minus'' comes from a contribution from the derived pushforward). 
There is now a serious impropriety if we try to forget the
$\CC^\times$ action, projecting out the vertical direction.

This could be fixed by replacing these measures with Fourier equivalent
ones pointing rightward. Our conjecture within theorem \ref{thm:main}
is that in the case of BB decompositions (which does not include 
this example), that replacement is unnecessary.

\subsection*{A non-Morse decomposition}

Consider the decomposition of $\CP^2 = \{[x,y,z]\}$
into 
$$ \{xyz\neq 0\} \coprod \{x=0,y\neq 0\} \coprod \{y=0,z\neq 0\}
\coprod \{z=0,x\neq 0\}.$$
If we follow the proof of theorem \ref{thm:BG}, 
but use this decomposition, we get the following equality of measures \\
\centerline{\epsfig{file=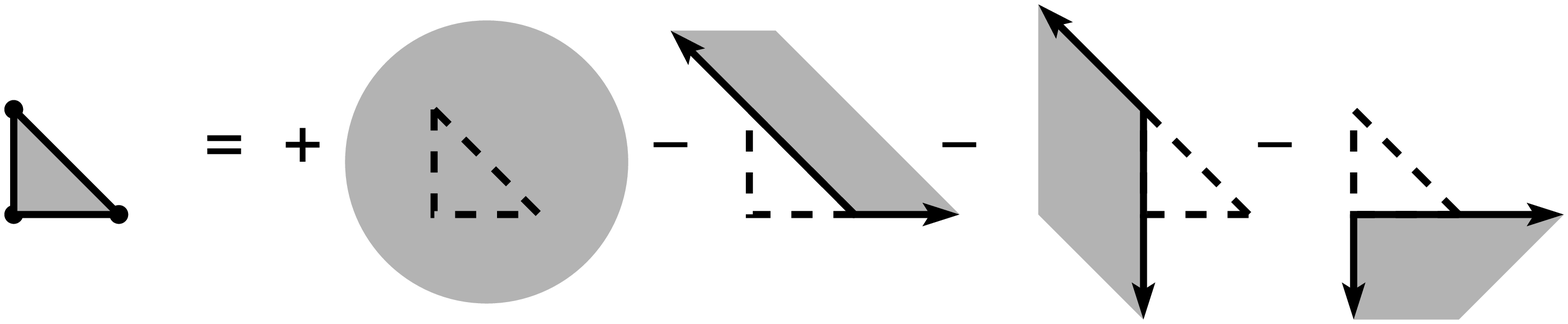,height=1in}} \\
which of course one could obtain by partial cancelation of the
Brianchon-Gram formula we drew after theorem \ref{thm:BG}.

We didn't here discuss nonabelian versions of our results, which
hopefully would allow for a similar geometric interpretation of \cite{Paradan}.
It is worth pointing out that when $T \into G \actson (M,\omega)$, 
it is frequently possible that $G$'s moment map is proper even though
$T$'s isn't (e.g. $T^1 \into SU(2) \actson \CC^2$), so one can't 
obviously derive the nonabelian from the abelian.

We end with a question, another conjecture, and an example. 

\begin{quote}
  Q. Let $M$ be 
  a smooth complex projective variety $M$ with a $T$-action,
  and $A\subseteq M$ a locally closed $T$-invariant
  smooth subvariety. What condition on $A$ guarantees that
  the projection of $supp(DH_{T\times \CC^\times}(cc(A)\subseteq T^*M,\omega_+))$
  to $\lie{t}^*$ is proper?
\end{quote}

There are two issues to be wary of -- higher cohomology involved in
defining $cc(A)$, and improperness of $A$ itself. The following is
an attempt to deal with each of those:

\newcommand\barA{{\overline A}}

\begin{quote}
  {\em Conjecture.} Assume $\barA \setminus A$ (the points added in the
  closure) supports an ample Cartier divisor in $\barA$.
  Assume that there exists an open $T$-invariant subset $U\subseteq CA$ of the 
  conormal bundle to $A$ such that $U/T$ is a proper scheme. 
  Then the projection in the question above is proper.
\end{quote}

We describe an example that would be covered by this conjecture. 
Consider the flag manifold $GL(3)/B$, its divisor 
$X^{231} := \overline{B r_1 r_2 B}/B$, and its rotations $c X^{231}, c^2 X^{231}$
where $c = \begin{bmatrix}  0& 1 & 0 \\ 0 &0& 1\\ 1&0&0 \end{bmatrix}$.
When we intersect those divisors pairwise, we get three $\PP^1$s.
Each stratum (the open stratum, three $\PP^1 \times \Gm$, three $\PP^1$)
has trivial normal bundle.
The analogue of theorem \ref{thm:BG} for this example,
computing the usual piecewise-linear measure on the hexagon,
looks as follows:

\epsfig{file=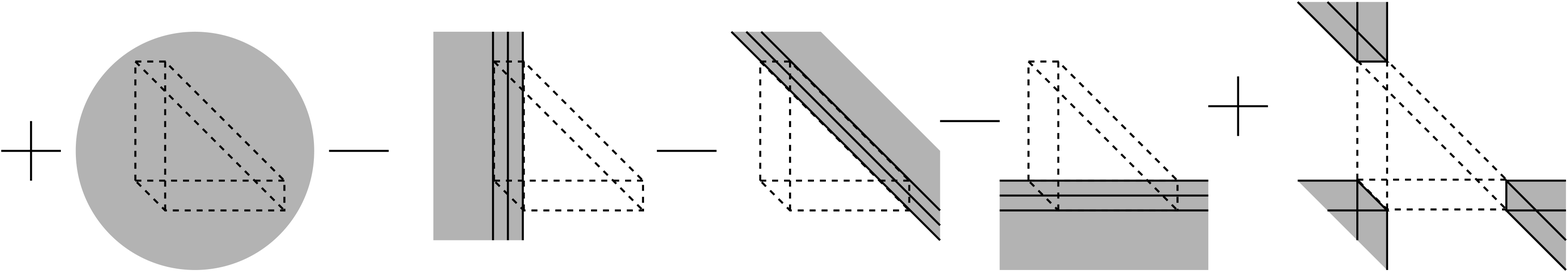,width=6in}

$$
\begin{matrix}
  +&
  \begin{matrix}    \text{a multiple of} \\    \text{Lebesgue measure} \\
    \text{on the whole plane} \\    \text{from open stratum}
  \end{matrix}
  &-&
  \begin{matrix}
    \text{a piecewise} \\
    \text{linear measure} \\
    \text{on a half-plane} \\
    \text{from a divisor}
  \end{matrix}
  &-&
  \begin{matrix}
    \text{another}
  \end{matrix}
  &-&
  \begin{matrix}
    \text{another } 
  \end{matrix}   +&
  \begin{matrix}
    \text{a piecewise} \\
    \text{linear measure} \\
    \text{on three half-planes} \\    \text{from three curves}
  \end{matrix} 
\end{matrix}
$$
This gives a manifestly $S_3$-invariant formula for the measure.

\bibliographystyle{alpha}    

\end{document}